\documentclass[10pt]{amsart}

\usepackage{mathrsfs}
\usepackage{amstext}
\usepackage{amscd}
\usepackage{latexsym}
\usepackage{amsfonts}
\usepackage{amssymb}
\usepackage{amsmath}
\usepackage{amsthm}
\usepackage{color}
\usepackage{multicol}
\usepackage[pdftex]{graphicx}
\graphicspath{{./pict/}}            
\DeclareGraphicsExtensions{.pdf,.png,.jpg}

\usepackage[cp1251]{inputenc}
\usepackage[russian,english]{babel}
\theoremstyle{plain}
\newtheorem{theorem}{Theorem}[section]

\newtheorem{corol}[theorem]{Corollary}

\theoremstyle{definition}
\newtheorem{definition}[theorem]{Definition}

\newtheorem{remark}[theorem]{Remark}

\begin{document}
\title[Matrices with self-interlacing spectrum]
 {Self-interlacing polynomials II:\\ Matrices with self-interlacing spectrum}
\author[M.~Tyaglov]{Mikhail Tyaglov}

\address{School of Mathematical Sciences, Shanghai Jiao Tong University\\
and  Faculty of Mathematics, Far East Federal University}
\email{tyaglov@sjtu.edu.cn}





\begin{abstract}
An $n\times n$ matrix is said to have a self-interlacing spectrum if its eigenvalues $\lambda_k$, $k=1,\ldots,n$, are distributed
as follows
$$
\lambda_1>-\lambda_2>\lambda_3>\cdots>(-1)^{n-1}\lambda_n>0.
$$
A method for constructing sign definite matrices with self-interlacing spectra from totally nonnegative ones is presented.
We apply this method to bidiagonal and tridiagonal matrices. In particular,
we generalize a result by O.\,Holtz on the spectrum of real symmetric
anti-bidiagonal matrices with positive nonzero entries.
\end{abstract}

\maketitle

\setcounter{equation}{0}

\section{Introduction}\label{section:Introduction}

In~\cite{Tyaglov_SI} there were introduced the so-called self-interlacing polynomials. A polynomial
$p(z)$ is called self-interlacing if all its roots are real, semple and interlacing the roots of the polynomial $p(-z)$.
It is easy to see that if $\lambda_k$, $k=1,\ldots,n$, are the roots of a self-interlacing polynomial, then the are distributed
as follows
\begin{equation}\label{SI_matrix.real.spectra}
\lambda_1>-\lambda_2>\lambda_3>\cdots>(-1)^{n-1}\lambda_n>0,
\end{equation}
or
\begin{equation}\label{SI_matrix.real.spectra.2}
-\lambda_1>\lambda_2>-\lambda_3>\cdots>(-1)^{n}\lambda_n>0.
\end{equation}

The polynomials whose roots are distributed as in~\eqref{SI_matrix.real.spectra} (resp. in~\eqref{SI_matrix.real.spectra.2}) are called self-interlacing
of kind $I$ (resp. of kind $II$). It is clear that a polynomial $p(z)$ is self-interlacing of kind $I$ if, and only if,
the polynomial $p(-z)$ is self-interlacing of kind $II$. Thus, it is enough to study self-interlacing of kind $I$, since
all the results for self-interlacing of kind $II$ will be obtained automatically.

\begin{definition}
An $n\times n$ matrix is said to possess a \textit{self-interlacing spectrum} if its eigenvalues $\lambda_k$, $k=1,\ldots,n$, are real, simple,
are distributed as in~\eqref{SI_matrix.real.spectra}.
\end{definition}

In~\cite{Tyaglov_SI} it was proved that a polynomial
\begin{equation}\label{main.poly}
p(z)=a_0z^n+a_1z^{n-1}+a_2z^{n-2}+a_3z^{n-3}+a_4z^{n-4}+\cdots+a_n=\sum\limits_{k=0}^na_kz^{n-k}
\end{equation}
is self-interlacing of kind $I$ if, and only if, the polynomial
\begin{equation}\label{main.poly.2}
p(z)=a_0z^n-a_1z^{n-1}-a_2z^{n-2}+a_3z^{n-3}+a_4z^{n-4}+\cdots+(-1)^{\tfrac{n(n+1)}{2}}a_n=\sum\limits_{k=0}^n(-1)^{\tfrac{k(k+1)}{2}}a_kz^{n-k}
\end{equation}
is Hurwitz stable, that is, has all its roots in the open left half-plane of the complex plane. Thus, there is a one-to-one correspondence between
the self-interlacing polynomials of kind $I$ and the Hurwitz stable polynomials. Now since the set of all polynomials with \textit{positive} roots
is isomorphic to a subset of the set of Hurwitz stable polynomials (to the set of all polynomials with negative roots), we conclude that this set
is isomorphic to a subset of the set of all self-interlacing polynomials of kind $I$. Consequently, it is worth to relate some classes of positive
definite and totally nonnegative matrices to matrices with self-interlacing spectrum.

In this work, we consider some classes of
\textit{real} matrices with self-interlacing spectrum and develop
a method of constructing such kind of matrices from a given
totally positive matrix. Namely, we show how and under what conditions 
it is possible to relate a totally nonnegative matrix with a matrix with
self-interlacing spectrum (Theorem~\ref{teorem.oscill_signreg.matrix.1}).
We apply this theorem to totally nonnegative bidiagonal and tridiagonal
matrices (generalizing a result by O.~Holtz) and explain how our technique 
can be extended for other classes of structured matrices. A part of this 
work was appeared first in the technical report~\cite{Tyaglov_GHP}.

\setcounter{equation}{0}

\section{Matrices with self-interlacing
spectrum}\label{section:matrices.SI}

At first, we recall some definitions
and statements from the book~\cite[Chapter~V]{KreinGantmaher}.

\begin{definition}[\cite{KreinGantmaher}]\label{def.sign definite.matrix}
A square matrix $A=\|a_{ij}\|_1^n$ is called \textit{sign
definite} of class~$n$ if for any $k\leqslant n$, all the non-zero
minors of order $k$ have the same sign $\varepsilon_k$. The
sequence $\{\varepsilon_1,\varepsilon_2,\ldots,\varepsilon_n\}$ is
called the \textit{signature sequence} of the matrix $A$.

A sign definite matrix of class $n$ is called \textit{strictly
sign definite} of class $n$ if all its minors are different from
zero.
\end{definition}
\begin{definition}[\cite{KreinGantmaher}]
A square sign definite matrix $A=\|a_{ij}\|_1^n$ of class $n$ is
called the \textit{matrix of class}~$n^{+}$ if some its power is a
strictly sign definite matrix of class $n$.
\end{definition}

Note that a sign definite (strictly sign definite) matrix of class
$n$ with the signature sequence
$\varepsilon_1=\varepsilon_2=\ldots=\varepsilon_n=1$ is totally
nonnegative (\textit{strictly totally positive}). Also a sign
definite matrix of class~$n^{+}$ with the signature sequence of
the form $\varepsilon_1=\varepsilon_2=\ldots=\varepsilon_n=1$ is
an \textit{oscillatory} (oscillation) matrix (see~\cite{KreinGantmaher}), that
is, a totally nonnegative matrix whose certain power is strictly
totally positive. It is clear from the Binet-Cauchy formula~\cite{Gantmakher.1} that
the square of a sign definite matrix is totally nonnegative.

In~\cite[Chapter~V]{KreinGantmaher} it was established the following theorem.
\begin{theorem}\label{Thm.SI_spectra.matrix}
Let the matrix $A=\|a_{ij}\|_1^n$ be totally nonnegative. Then the
matrices $B=\|a_{n-i+1,j}\|_1^n$ and $C=\|a_{i,n-j+1}\|_1^n$ are
sign definite of class $n$. Moreover, the signature sequence of
the matrices $B$ and~$C$ is as follows:
\begin{equation}\label{SI_matrix.eps}
\varepsilon_k=(-1)^{\tfrac{k(k-1)}2},\quad k=1,2,\ldots,n.
\end{equation}
\end{theorem}

Note that the matrices $B$ and $C$ can be represented as follows
\begin{equation*}
B=JA,\qquad\text{and}\qquad C=AJ,
\end{equation*}
where
\begin{equation}\label{Matrice.J}
J=
\begin{pmatrix}
    0 & 0 &\dots& 0 & 1 \\
    0 & 0 &\dots& 1 & 0 \\
    \vdots&\vdots&\cdot&\vdots&\vdots\\
    0 & 1 &\dots& 0 & 0 \\
    1 & 0 &\dots& 0 & 0 \\
\end{pmatrix}.
\end{equation}

It is easy to see that the matrix $J$ is sign definite of class
$n$ (but not of class $n^{+}$) with the signature sequence of the
form~\eqref{SI_matrix.eps}. So by the Binet-Cauchy
formula we obtain the following statement.
\begin{theorem}\label{teorem.SI_oscill.spectra.matrix}
The matrix $A=\|a_{ij}\|_1^n$ is totally nonnegative if and only
if the matrix $JA$ (or the matrix~$AJ$) is sign definite of class
$n$ with the signature sequence~\eqref{SI_matrix.eps} where the matrix
$J$ is defined in~\eqref{Matrice.J}.
\end{theorem}

Obviously, the converse statement is also true.
\begin{theorem}
The matrix $A=\|a_{ij}\|_1^n$ is a sign definite matrix of class
$n$ with the signature sequence~\eqref{SI_matrix.eps} if and only
if the matrix $JA$ (or the matrix~$AJ$) is totally nonnegative.
\end{theorem}


In the sequel, we need the following two theorems established in
the book~\cite[Chapter~V]{KreinGantmaher}.
\begin{theorem}\label{teorem.signreg.matrix}
Let the matrix $A=\|a_{ij}\|_1^n$ be a sign definite of class
$n^{+}$ with the signature sequence $\varepsilon_k$,
$k=1,2,\ldots,n$. Then all the eigenvalues $\lambda_k$,
$k=1,2,\ldots,n$, of the matrix $A$ are nonzero real and simple,
and if
\begin{equation}\label{SI_matrix.modules}
|\lambda_1|>|\lambda_2|>\ldots>|\lambda_n|>0,
\end{equation}
then
\begin{equation}\label{SI_matrix.signs.eigvals}
\emph{sign}\,\lambda_k=\dfrac{\varepsilon_k}{\varepsilon_{k-1}},\qquad
k=1,2,\ldots,n,\,\varepsilon_0=1.
\end{equation}
\end{theorem}
\begin{theorem}\label{teorem.oscill.matrix}
A totally nonnegative matrix $A=\|a_{ij}\|_1^n$ is oscillatory if,
and only if, $A$ is nonsingular, and the following inequalities hold
\begin{equation*}
a_{j,j+1}>0,\qquad\text{and}\qquad a_{j+1,j}>0,\qquad
j=1,2,\ldots,n-1.
\end{equation*}
\end{theorem}

Now we are in a position to complement
Theorem~\ref{Thm.SI_spectra.matrix}.
\begin{theorem}\label{teorem.oscill_signreg.matrix.1}
Let all the entries of a nonsingular matrix $A=\|a_{ij}\|_1^n$ be
nonnegative, and let for each $i$,
$i=1,2,\ldots,n-1$, there exist numbers $r_1$ and $r_2$,
$1\leqslant r_1,r_2\leqslant n$, such that
\begin{equation}\label{matrix.entries.1}
a_{n-i,r_1}\cdot a_{n+1-r_1,i}>0,\qquad a_{n+1-i,r_2}\cdot
a_{n+1-r_2,i+1}>0
\end{equation}
%
%
\begin{equation}\label{matrix.entries.2}
(\text{or}\qquad a_{i,n+1-r_1}\cdot a_{r_1,n-i}>0,\qquad
a_{i+1,n+1-r_2}\cdot a_{r_2,n+1-i}>0).
\end{equation}
The matrix $A$ is totally nonnegative if and only if the matrix
$B=JA=\|a_{n-i+1,j}\|_1^n$ \emph{(}or, respectively, the matrix
$C=AJ=\|a_{i,n-j+1}\|_1^n$\emph{)} is sign definite of class~$n^{+}$ with the
signature sequence defined in~\eqref{SI_matrix.eps}. Moreover, the
matrix~$B$ \emph{(}or, respectively, the matrix~$C$\emph{)}
possesses a self-interlacing spectrum.
\end{theorem}
\begin{proof}
We prove the theorem in the case when the
condition~\eqref{matrix.entries.1} holds. The case of the
condition~\eqref{matrix.entries.2} can be established analogously.

Let $A$ be a nonsingular totally nonnegative matrix, and let the
condition~\eqref{matrix.entries.1} hold. From
Theorem~\ref{teorem.SI_oscill.spectra.matrix} it follows that the
matrix $B=JA$ is sign definite of class~$n$ with the signature
sequence~\eqref{SI_matrix.eps}. In order to the matrix to be sign
definite of class~$n^{+}$ it is necessary and sufficient that a
certain power of this matrix $B$ be strictly sign definite of
class~$n$. Since the entries of the matrix $J$ have the form
\begin{equation*}
(J)_{ij}=\begin{cases}
&1,\qquad i=n+1-j;\\
&0,\qquad i\neq n+1-j;
\end{cases}
\end{equation*}
the entries of the matrix $B$ can be represented as follows:
$$
b_{ij}=\sum\limits_{k=1}^{n}(J)_{ik}a_{kj}=a_{n+1-i,j}.
$$

Consider the totally nonnegative matrix $B^2$. Its entries have
the form
\begin{equation*}
(B^2)_{ij}=\sum\limits_{k=1}^{n}b_{ik}b_{kj}=\sum\limits_{k=1}^{n}a_{n+1-i,k}a_{n+1-k,j}.
\end{equation*}
From these formul\ae~and from~\eqref{matrix.entries.1} it follows
that all the entries of the matrix $B^2$ above and under the main
diagonal are positive, that is, $(B^2)_{i,i+1}>0$ and
$(B^2)_{i+1,i}>0$, $i=1,2,\ldots,n-1$. By
Theorem~\ref{teorem.oscill.matrix}, $B^2$ is an oscillatory
matrix. According to the definition of oscillatory matrices, a
certain power of $B^2$ is strictly totally positive. Thus, we
proved that a certain power of the matrix $B$ is strictly sign
definite, so $B$ is a sign definite matrix of class $n^{+}$ with
the signature sequence~\eqref{SI_matrix.eps} according to
Theorem~\ref{Thm.SI_spectra.matrix}. By
Theorem~\ref{teorem.signreg.matrix}, all eigenvalues of the matrix
$B$ are nonzero real and simple. Moreover, if we enumerate the
eigenvalues in order of decreasing absolute values as
in~\eqref{SI_matrix.modules}, then
from~\eqref{SI_matrix.signs.eigvals} and~\eqref{SI_matrix.eps} we
obtain that the spectrum of $B$ is of the
form~\eqref{SI_matrix.real.spectra}.

The converse assertion of the theorem follows from
Theorem~\ref{teorem.SI_oscill.spectra.matrix}.
\end{proof}

\begin{remark}
If the matrix $-A$ is totally nonnegative and the
conditions~\eqref{matrix.entries.1} (or the
conditions~\eqref{matrix.entries.2}) hold, then the matrix $B=JA$
(or, respectively, $C=AJ$) has a spectrum of the
form~\eqref{SI_matrix.real.spectra.2}.
\end{remark}
\begin{remark}\label{remark.oscill_signreg.matrix.1}
One can obtain another types of totally nonnegative matrices which
result matrices with self-interlacing spectra after multiplication
by the matrix $J$. To do this we need to change the
conditions~\eqref{matrix.entries.1}--\eqref{matrix.entries.2} by
another ones such that, for instance, the matrix $B^4$ (or $B^6$,
or $B^8$ etc) becomes oscillatory.
\end{remark}

Theorem~\ref{teorem.oscill_signreg.matrix.1} implies the following
corollary.
\begin{corol}\label{corol.oscill_signreg.matrix.1}
A nonsingular matrix $A$ with positive entries is totally
nonnegative if, and only if, the matrix $B=JA$ (or the matrix
$C=AJ$) is sign definite of class~$n^{+}$ with the signature
sequence~\eqref{SI_matrix.eps}. Moreover, the matrix $B$ possesses a
self-interlacing spectrum.
\end{corol}

Consider a partial case of conditions~\eqref{matrix.entries.1}.
Suppose that all diagonal entries of the matrix $A$ are positive:
$a_{jj}>0$, $j=1,2,\ldots,n$. It is easy to see that in this case
the conditions~\eqref{matrix.entries.1} hold if $a_{j,j+1}>0$,
$j=1,2,\ldots,n-1$. If all remaining entries of the matrix $A$ are
nonnegative, then by Theorem~\ref{teorem.oscill_signreg.matrix.1}
the matrices $B=JA$ and $C=AJ$ have self-interlacing spectra whenever
$A$ is totally nonnegative.

If all other entries of the matrix $A$ (that is, all entries
except $a_{jj}$ and $a_{j,j+1}$, which are positive) equal zero,
then $A$ is a bidiagonal matrix with positive entries on and under
the main diagonal. Clearly, $A$ is totally nonnegative. Then by
Theorem~\ref{teorem.oscill_signreg.matrix.1} the matrix $B=JA$ possesses
a self-interlacing spectrum. Note that the matrix $B$ in this case
is anti-bidiagonal with positive entries, that is, it has the form
\begin{equation}\label{Matrix.antybidiag}
B=
\begin{pmatrix}
    0 & 0 &0&&\dots&&  0   & 0 & b_n \\
    0 &0&0&&\dots&&  0 & b_{n-2} & b_{n-1} \\
     0  & 0&0&&\dots&& b_{n-4}& b_{n-3}& 0 \\
     & & &  & &\cdot & & & \\
    \vdots&\vdots&\vdots&&a&&\vdots&\vdots&\vdots\\
    & & & \cdot & & & & & \\
     0& 0 &c_{n-4}&&\dots&&0&0&0\\
      0 &c_{n-2}&c_{n-3}&&\dots&&0&0&0\\
     c_n&c_{n-1}&  0  &&\dots&&0&0&0\\
\end{pmatrix},\qquad b_j>0,\,\,c_j>0,\,\,a>0.
\end{equation}
where all the entries $b_j$,
$j=2,3,\ldots,n$, lie above the main diagonal, all the entries
$c_j$, $j=2,3,\ldots,n$, lie under the main diagonal and the
only entry on the main diagonal is $a$.

Thus we proved the following fact.
\begin{theorem}
Any anti-bidiagonal matrix with positive entries as in~\eqref{Matrix.antybidiag} possesses
a self-interlacing spectrum.
\end{theorem}
In~\cite{H4} the same theorem was established under the additional assumption that $b_j=c_j$, $j=2,\ldots,n$.
In the same manner as in the work~\cite{H4}, it also can be shown that the spectrum of the
matrix~\eqref{Matrix.antybidiag} coincides with the spectrum of
the following tridiagonal matrix
\begin{equation*}
T=
\begin{pmatrix}
    a_1 & b_2 &  0 &\dots&   0   & 0 \\
    c_2 & 0 &b_3 &\dots&   0   & 0 \\
     0  &c_3 & 0 &\dots&   0   & 0 \\
    \vdots&\vdots&\vdots&\ddots&\vdots&\vdots\\
     0  &  0  &  0  &\dots&0&b_n\\
     0  &  0  &  0  &\dots&c_n&0\\
\end{pmatrix}.
\end{equation*}
It is well-known~\cite{KreinGantmaher} that the spectrum of this
matrix does not depend on the entries $b_j$ and $c_j$ separately.
It depends on products $b_jc_j$, $j=2,3,\ldots,n$. So in order for
the matrices~\eqref{Matrix.antybidiag} and $T$ to have a
self-interlacing spectra, it is sufficient that the inequalities
$a_1>0$ and $b_jc_j>0$, $j=2,3,\ldots,n$, hold. However, the inverse 
spectral problem cannot be solved here uniquely unless $b_j=c_j$, $j=2,\ldots,n$ (the case of the work~\cite{H4}).

Finally, consider a tridiagonal matrix
\begin{equation}\label{M_J}
M_{J}=
\begin{pmatrix}
    a_1 & b_1 &  0 &\dots&   0   & 0 \\
    c_1 & a_2 &b_2 &\dots&   0   & 0 \\
     0  &c_2 & a_3 &\dots&   0   & 0 \\
    \vdots&\vdots&\vdots&\ddots&\vdots&\vdots\\
     0  &  0  &  0  &\dots&a_{n-1}&b_{n-1}\\
     0  &  0  &  0  &\dots&c_{n-1}&a_n\\
\end{pmatrix},
\end{equation}
where $a_k,b_k,c_k\in\mathbb{R}$ and $c_kb_k\neq0$.
In~\cite{KreinGantmaher}, there was proved the following fact.
\begin{theorem}\label{teorem.SI_Matrix.Jacobi}
Let the matrix $M_{J}$ defined in~\eqref{M_J} be nonnegative. Then $M_J$ is oscillatory if, and only if, all the entries
$b_k$ and $c_k$ are positive and all the leading principal minors
of $M_{J}$ are also positive:
\begin{equation}\label{Jacobi.ineq}
\begin{vmatrix}
    a_1 & b_1 &  0 &\dots&   0   & 0 \\
    c_1 & a_2 &b_2 &\dots&   0   & 0 \\
     0  &c_2 & a_3 &\dots&   0   & 0 \\
    \vdots&\vdots&\vdots&\ddots&\vdots&\vdots\\
     0  &  0  &  0  &\dots&a_{k-1}&b_{k-1}\\
     0  &  0  &  0  &\dots&c_{k-1}&a_k\\
\end{vmatrix}>0,\qquad k=1,\ldots,n.
\end{equation}
\end{theorem}

This theorem together with
Theorem~\ref{teorem.oscill_signreg.matrix.1} implies the following
statement.
\begin{theorem}\label{teorem.anti.tridiag}
The anti-tridiagonal matrix
\begin{equation*}
A_{J}=
\begin{pmatrix}
    0 & 0 &0&\dots&  0   & b_1 & a_1 \\
    0 &0&0&\dots&  b_2 & a_2 & c_1 \\
     0  & 0&0&\dots& a_3& c_2& 0 \\
    \vdots&\vdots&\vdots&\ddots&\vdots&\vdots\\
     0&b_{n-2}&a_{n-2}&\dots&0&0&0\\
     b_{n-1}&a_{n-1}&c_{n-2}&\dots&0&0&0\\
     a_n&c_{n-1}&  0  &\dots&0&0&0\\
\end{pmatrix},
\end{equation*}
where $a_j,b_j,c_j>0$ for $j=1,2,\ldots,n-1$, is sign definite of
class $n^{+}$ and possesses a self-interlacing
spectrum if, and only if, the
following inequalities hold:
\begin{equation}\label{antyJacobi.ineq}
(-1)^{\frac{k(k-1)}2}
\begin{vmatrix}
    0 & 0 &0&\dots&  0   & b_1 & a_1 \\
    0 &0&0&\dots&  b_2 & a_2 & c_1 \\
     0  & 0&0&\dots& a_3& c_2& 0 \\
    \vdots&\vdots&\vdots&\ddots&\vdots&\vdots\\
     0&b_{k-2}&a_{k-2}&\dots&0&0&0\\
     b_{k-1}&a_{k-1}&c_{k-2}&\dots&0&0&0\\
     a_k&c_{k-1}&  0  &\dots&0&0&0\\
\end{vmatrix}>0,\quad k=1,2,\ldots,n.
\end{equation}
\end{theorem}
\begin{proof}
If the matrix $A_{J}$ is sign definite of class $n^{+}$ and has a
self-interlacing spectrum of the
form~\eqref{SI_matrix.real.spectra}, then according to
Theorem~\ref{teorem.signreg.matrix}, the signs of its nonzero
minors can be calculated by the formula~\eqref{SI_matrix.eps}.
This implies the inequalities~\eqref{antyJacobi.ineq}.

Conversely, let the inequalities~\eqref{antyJacobi.ineq} hold.
Then we have that the inequalities~\eqref{Jacobi.ineq} hold for
the~matrix $M_J=JA_J$. By Theorem~\ref{teorem.SI_Matrix.Jacobi},
the matrix $M_J$ is oscillatory and, in particular, totally
nonnegative~\cite{KreinGantmaher}. Now notice that $(M_J)_{ii}>0$, $i=1,\ldots,n$, and
$(M_J)_{k,k+1}>0$, $k=1,\ldots,n-1$, so $M_J$ satisfies the
condition~\eqref{matrix.entries.1} of
Theorem~\ref{teorem.oscill_signreg.matrix.1}. Therefore, the
matrix $A_J$ is sign definite of class $n^{+}$ and has a
self-interlacing spectrum of the
form~\eqref{SI_matrix.real.spectra}.
\end{proof}

\section{Conclusion}

Theorem~\ref{teorem.oscill_signreg.matrix.1} and Corollary~\ref{corol.oscill_signreg.matrix.1}
provide a method of constructing matrices with self-interlacing
spectrum from given totally nonnegative and oscillatory matrices. We gave two examples of application
of this method for bidiagonal and tridiagonal matrices. One can use Remark~\ref{remark.oscill_signreg.matrix.1}
to generalize Theorem~\ref{teorem.oscill_signreg.matrix.1} and to apply the generalization for
other types of structured matrices.

%

\subsection*{Acknowledgment}
The author is Shanghai Oriental Scholar whose work was supported by Russian Science Foundation, grant no. 14-11-00022.


\end{document}